\documentclass[10pt]{amsart}

 \usepackage{graphicx,color,amssymb, amsmath, amsthm, mathrsfs}

\usepackage[all]{xy}
\usepackage{amssymb}
\usepackage{slashed}
\usepackage{mathtools}

\swapnumbers
\newtheorem{theorem}{Theorem}[section]
\newtheorem{lemma}[theorem]{Lemma}
\newtheorem{proposition}[theorem]{Proposition}
\newtheorem{corollary}[theorem]{Corollary}
\theoremstyle{definition}
\newtheorem{definition}[theorem]{Definition}

\newtheorem{remark}[theorem]{Remark}

\newtheorem*{remark*}{Remark}
\newtheorem*{remarks*}{Remarks}
\newtheorem*{definition*}{Definition}

\usepackage{calrsfs}

\usepackage[T1]{fontenc} 
\usepackage{textcomp}
\usepackage{times}
\usepackage[scaled=0.92]{helvet} 
\usepackage{amsmath,amssymb,amsfonts,amsthm}

\newtheorem{notn}{Notation}

\begin{document}
\title{Spectral triples on $O_N$}
\author[M. Goffeng]{Magnus Goffeng}
\address{\normalfont{Department of Mathematical Sciences, Chalmers University of Technology and the University of Gothenburg, SE-412 96 Gothenburg, Sweden}}
\email{goffeng@chalmers.se}
\author[B. Mesland]{Bram Mesland}
\address{\normalfont{Institut f\"{u}r Analysis, Leibniz Universit\"{a}t Hannover, Welfengarten 1, 30167,
Hannover, Germany}}
\email{mesland@math.uni-hannover.de}

\begin{abstract}
We give a construction of an odd spectral triple on the Cuntz algebra $O_{N}$, whose $K$-homology class generates the odd $K$-homology group $K^1(O_{N})$. Using a metric measure space structure on the Cuntz-Renault groupoid, we introduce a singular integral operator which is the formal analogue of the logarithm of the Laplacian on a Riemannian manifold. Assembling this operator with the infinitesimal generator of the gauge action on $O_{N}$ yields a $\theta$-summable spectral triple whose phase is finitely summable. The relation to previous constructions of Fredholm modules and spectral triples on $O_{N}$ is discussed.
\end{abstract}

\maketitle
\section{Introduction}
\label{sec:1}

We give a geometrically inspired construction of spectral triples on the Cuntz algebra $O_N$ with non-trivial $K$-homological content. One reason such spectral triples have been elusive is Connes' construction of traces from finitely summable spectral triples \cite{Connestrace}. Purely infinite $C^*$-algebras such as $O_{N}$ are traceless and should thus be viewed as infinite dimensional objects, at best carrying $\theta$-summable spectral triples. Another difficulty is presented by the fact that the $K$-homology of $O_{N}$ is torsion, so the index pairing cannot be used to detect $K$-homology classes.

In the literature, several approaches to noncommutative geometry on Cuntz-Krieger algebras have been explored. The crossed product $C^{*}$-algebra associated to the action of a free group on its Gromov boundary gives rise to a Cuntz-Krieger algebra, and in \cite{CMRV} that geometric picture is used to establish the existence of $\theta$-summable spectral triples on such $C^{*}$-algebras. On the other hand, twisted noncommutative geometries \cite{cmtwisted} circumvent the obstruction to finite summability whereas semifinite noncommutative geometries \cite{paskrennie} allow for the extraction of index-theoretic invariants.

Recent years have seen explicit constructions of spectral triples on Cuntz-Krieger algebras \cite{GM} and more generally on Cuntz-Pimsner algebras \cite{GMR}, originating in the dynamics of subshifts of finite type. Their classes in $K$-homology were computed in \cite{GM} using Poincar\'{e} duality and extension theory (see \cite{kaminkerputnam,kasparov}), thus bypassing the difficulties discussed above. These spectral triples have the remarkable feature that they are $\theta$-summable, but their bounded transforms $\chi(D)$, using a suitably chosen function $\chi\in C_b(\mathbb{R})$ such that $\lim_{t\to \pm\infty}\chi(t)=\pm 1$, are finitely summable. Providing a geometric context and understanding the distinct dimensional behaviours of bounded and unbounded Fredholm modules over $O_N$ is the main problem motivating this paper.

Using the metric and Patterson-Sullivan measure on the full $N$-shift, we equip the Cuntz-Renault groupoid with the structure of a metric measure space. Then we consider a singular integral kernel formally similar to that of the logarithm of the Laplacian on a closed Riemannian manifold. We explicitly relate the associated integral operator to the depth-kore operator from \cite{GMR}, yielding a geometric construction of a $K$-homologically non-trivial noncommutative geometry on $O_N$. 

\section{Statement of results on $O_N$}

Before stating our results, we recall several notions from noncommutative geometry. The reader familiar with summability properties in noncommutative geometry and the groupoid model of $O_N$ can proceed to page \pageref{mainthm} for the main results. 

Let $A$ be a unital $C^*$-algebra. A spectral triple is a triple $(A,\mathcal{H},D)$ where $A$ acts unitally on the Hilbert space $\mathcal{H}$ and $D$ is a self-adjoint operator with compact resolvent on $\mathcal{H}$ such that 
$$\mathrm{Lip}_D(A):=\{a\in A: \; a\mathrm{Dom}(D)\subseteq \mathrm{Dom}(D)\;\mbox{  and  }\;[D,a]\;\mbox{  is bounded}\}\subseteq A\;\mbox{  is dense.}$$
A spectral triple is sometimes called an unbounded Fredholm module. A bounded Fredholm module is a triple $(A,\mathcal{H},F)$ as above safe the fact that $F$ is a bounded operator assumed to satisfy that $F^2-1, F-F^*, [F,a]\in \mathbb{K}(\mathcal{H})$ for any $a\in A$. 

Dimensional properties of (un)bounded Fredholm modules are described in terms of operator ideals. For a compact operator $T$ on a Hilbert space $\mathcal{H}$, we denote by $\mu_{k}(T)$ its sequence of singular values. Given $p\in (0,\infty)$, let $\mathcal{L}^{p}(\mathcal{H}),$ denote the $p$-th Schatten ideal and $\mbox{Li}^{1/p}(\mathcal{H})\subset\mathbb{K}(\mathcal{H})$ the symmetrically normed ideal defined by
$$\mbox{Li}^{1/p}(\mathcal{H}):=\{T\in \mathbb{K}(\mathcal{H}): \mu_{k}(T)=O((\log k)^{-1/p})\}. $$

An unbounded Fredholm module $(A,\mathcal{H},D)$ is said to be $p$-summable if $(D\pm i)^{-1}\in \mathcal{L}^p (\mathcal{H})$ and $\theta$-summable if $(D\pm i)^{-1}\in \mbox{Li}^{1/2} (\mathcal{H})$. Note that $\theta$-summability is equivalent to requiring that $e^{-tD^{2}}\in\mathcal{L}^{1}(\mathcal{H})$ for all $t>0$.  
A bounded Fredholm module $(A,\mathcal{H},F)$ is said to be $p$-summable if  $$F^2-1, F-F^*\in \mathcal{L}^{p/2}(\mathcal{H}), \quad \mbox{and }\quad [F,a]\in \mathcal{L}^p(\mathcal{H}),$$ 
and $\theta$-summable if
$$F^2-1, F-F^*\in \mbox{Li}(\mathcal{H}), \quad \mbox{and }\quad [F,a]\in \mbox{Li}^{1/2}(\mathcal{H}),$$ 
for all $a$ in a dense subalgebra of $A$. 

We emphasize the difference between the two definitions. Summability of an unbounded Fredholm module is a property of the operator $D$, whereas summability of a bounded Fredholm module is a property of the operator $F$ and of its commutators with the algebra $A$. The two notions are related as follows. If $(A,\mathcal{H},D)$ is a $p$-summable (resp. $\theta$-summable) unbounded Fredholm module, then $(A,\mathcal{H},\chi(D))$ is a $p$-summable (resp. $\theta$-summable) bounded Fredholm module if $\chi\in C_b(\mathbb{R})$ is a function satisfying $\chi^2=1+O(|x|^{-2})$ as $|x|\to \infty$. Conversely, a $\theta$-summable bounded Fredholm module can be lifted to a $\theta$-summable unbounded Fredholm module, see \cite[Chapter IV.8, Theorem 4]{connesbook}. This result fails for finite summability, as is shown in particular by the examples in this paper.

Any $K$-homology class on a Cuntz-Krieger algebra is represented by a finitely summable bounded Fredholm module \cite{GM}. In general, Cuntz-Krieger algebras admit no finitely summable spectral triples, as discussed above. This phenomenon is widespread and, for instance, occurs for boundary crossed product algebras of hyperbolic groups \cite{emersonnica}. The action of a free group on its Gromov boundary falls into the class of examples considered in both \cite{emersonnica} and \cite{GM}. To our knowledge, obstructions to finite summability at the bounded level have not been studied. At present, the example \cite[Lemma 6, page 95]{GM} of a $K$-homology class not admitting finitely summable bounded representatives is the only one known to the authors. 

Before stating our main results, we recall some facts about $O_N$ that we review in more detail in Section \ref{sec:2}. For $N>1$, the Cuntz algebra $O_N$ \cite{thealgebraon} is defined as the universal $C^*$-algebra generated by $N$ isometries with orthogonal ranges. As the $C^*$-algebra $O_N$ is simple, it can be constructed in any of its Hilbert space realizations. That is, for any operators $S_1,\ldots ,S_N$ such that $S_j^*S_k=\delta_{jk}$ and $1=\sum_{j=1}^NS_jS_j^*$, $O_N$ is canonically isomorphic to the $C^*$-algebra generated by $S_1,\ldots ,S_N$. 

An important realization of $O_N$ is as the groupoid $C^*$-algebra of the Cuntz-Renault groupoid $\mathcal{G}_N$ introduced in \cite[Section III.2]{Ren1}. The unit space of $\mathcal{G}_N$ is the full one-sided sequence space $\Omega_N:=\{1,\ldots, N\}^{\mathbb{N}}$. We equip $\Omega_N$ with the product topology in which it is compact and totally disconnected. Elements $x\in \Omega_N$ are written $x=x_1x_2\cdots$ where $x_j\in \{1,\ldots, N\}$. The shift $\sigma:\Omega_N\to \Omega_N$ is defined by $\sigma(x_1x_2x_3\cdots)=x_2x_3\cdots$ and is a surjective local homeomorphism. For a finite word $\mu=\mu_1\mu_2\cdots \mu_k\in \{1,\ldots, N\}^k$  we define the cylinder set 
$$C_\mu:=\{x\in \Omega_N: x=\mu x' \mbox{ for some } x'\in \Omega_N\}.$$ We call $|\mu|:=k$ the length of $\mu$. 
As a set, the Cuntz-Renault groupoid is given by
\begin{equation}
\label{defininggn}
\mathcal{G}_N:=
\{(x,n,y)\in \Omega_N\times \mathbb{Z}\times \Omega_N: 
\exists k \; \sigma^{n+k}(x)=\sigma^k(y)\}\rightrightarrows  \Omega_N,
\end{equation}
with domain map $d_{\mathcal{G}}:\mathcal{G}\to \Omega_N$, range map $r_{\mathcal{G}}:\mathcal{G}\to \Omega_N$ and product $\cdot$ defined by
$$d_\mathcal{G}(x,n,y):=y,\quad r_\mathcal{G}(x,n,y):=x,\quad (x,n,y)\cdot(y,m,z)=(x,n+m,z).$$ 
The space $\mathcal{G}_N$ admits an extended metric $\rho_\mathcal{G}$ defined below in Definition \ref{definingmetric}. The  \'{e}tale topology described in \cite{Ren2} coincides with the metric topology on $\mathcal{G}_{N}$ induced by $\rho_{\mathcal{G}}$ (see Section \ref{sec:2}, Proposition \ref{propertiesofrhog}).There is an isomorphism $O_N\cong C^*(\mathcal{G}_N)$ (see \cite{Ren1,Ren2}) and an expectation $\Phi:C^*(\mathcal{G}_N)\to C(\Omega_N)$ induced by the clopen inclusion \begin{equation}
\label{unitincl}
\Omega_{N}\subset\mathcal{G}_{N},\quad x\mapsto (x,0,x).
\end{equation}
The algebra $O_N$ admits a unique KMS-state $\phi$ (see Section \ref{subsec:2.2}), and we write $L^{2}(O_{N}):=L^{2}(O_{N},\phi)$ for its GNS-representation  (see below in Subsection \ref{subsec:2.2}). Under the isomorphism $O_N\cong C^*(\mathcal{G}_N)$ we have  $L^2(O_N)=L^2(\mathcal{G}_N,m_\mathcal{G})$  for the measure $m_{\mathcal{G}}:=d^*_\mathcal{G}m_\Omega$ induced by the Patterson-Sullivan measure $m_\Omega$ on $\Omega_N$, characterized by $m_{\Omega}(C_\mu):=N^{-|\mu|}$. We often write $g=(x,n,y)$ for an element of $\mathcal{G}_N$. Note that the Hausdorff dimension of $\Omega_{N}$, and hence of $\mathcal{G}_{N}$, equals $\log N$.

\begin{definition}
\label{deftones}
We define the densely defined operators $c$, $T$ and $P_\mathcal{F}$ on $L^2(O_N)$ as follows. 
\begin{enumerate}
\item Define $c_0$ by $\mathrm{Dom}(c_0)=C_c(\mathcal{G}_N)$ and $c_0f(x,n,y):=nf(x,n,y)$ and let $c$ denote the closure of $c_0$.
\item Define $T_0$ by letting $\mathrm{Dom}(T_0)$ be the compactly supported locally constant functions and 
$$T_0f(g):=\frac{1}{(1-N^{-1})} \int_{\mathcal{G}_N} \frac{f(g)-f(h)}{\rho_\mathcal{G}(g,h)^{\log(N)}}\mathrm{d}m_{\mathcal{G}}(h),$$
and let $T$ denote the closure of $T_0$. The extended metric $\rho_\mathcal{G}$ is defined below in Definition \ref{definingmetric}.
\item  Define the set 
\begin{equation}
\label{thesetxf}
X_{\mathcal{F}}:=\left\{(x,n,y)\in \Omega_N\times \mathbb{N}\times \Omega_N: \exists \mu\in \{1,\ldots, N\}^n \mbox{ s.t. } x\in C_\mu \mbox{  and  } \sigma^n(x)=y\right\}.
\end{equation}
Let $P_\mathcal{F}$ denote the integral operator on $L^2(\mathcal{G}_N)$ with integral kernel $\chi_{X_\mathcal{F}}$ (the characteristic function of $X_\mathcal{F}$).
\end{enumerate}
\end{definition}

There is an isomorphism $K^1(O_N)\cong \mathbb{Z}/(N-1)\mathbb{Z}$ defined from Poincar\'e duality for Cuntz-Krieger algebras \cite{kaminkerputnam} and the isomorphism $K_0(O_N)\cong \mathbb{Z}/(N-1)\mathbb{Z}$. We denote by $\widehat{[1]}\in K^{1}(O_{N})$ the class Poincar\'e dual to $[1]\in K_0(O_{N})$ and sometimes refer to this class as \emph{the generator} of $K^1(O_N)$. The generator of $K^1(O_N)$ is represented by the extension considered in \cite{evansonon}. In the sequel we will use the operator $T$ from Definition \ref{deftones} to construct spectral triples representing the $K$-homology class $\widehat{[1]}$. 

In the statement of our main result we will make use of the so called dispersion operator $B:L^2(O_N)\to L^2(O_N)$ which is a bounded operator defined below in Lemma \ref{displem} (see page \pageref{displem}). The dispersion operator measures how non-diagonal the operator $T$ is in a particular ON-basis of $L^2(O_N)$. We also make use of a certain projection $Q$ defined just before Theorem \ref{bigtcomp}.

\begin{theorem}
\label{mainthm}
The operators $c$, $T$ and $P_\mathcal{F}$ from Definition \ref{deftones} are well defined self-adjoint operators. In fact $P_\mathcal{F}$ is an orthogonal projection, $T$ is positive and  $D:=(2P_\mathcal{F}-1)|c|-T$ is a self-adjoint operator with compact resolvent. Moreover, 
\begin{enumerate}
\item $(O_N,L^2(O_N),D)$ is a spectral triple whose class coincides with $\widehat{[1]}\in K^1(O_N)$ and $\mathrm{e}^{-tD^2}$ is of trace class for all $t>0$, i.e. $D$ is $\theta$-summable.
\item Up to finite rank operators, $P_\mathcal{F}=\chi_{[0,\infty)}(D)$ and for any $p>0$, 
$$(O_N,L^2(O_N),2P_{\mathcal{F}}-1)$$ 
is a $p$-summable Fredholm module whose class is $\widehat{[1]}\in K^1(O_N)$.
\item The operator $\tilde{D}:=D-B+(N-1)^{-1}Q$ also defines a spectral triple on $O_N$, where $B$ is the dispersion operator (see Lemma \ref{displem}) and $Q$ is a projection (see before Theorem \ref{bigtcomp}). For any extended limit $\omega\in L^\infty[0,1]^*$ at $0$ there is a probability measure $\tilde{m}_\omega$ on $\Omega_N$ such that 
$$\tilde{\phi}_\omega(a):=\omega\left( \frac{\mathrm{Tr}(a\mathrm{e}^{-t\tilde{D}^2})}{\mathrm{Tr}(\mathrm{e}^{-t\tilde{D}^2})}\right), \quad a\in O_N,$$
is computed from $\tilde{\phi}_\omega(a)=\int_{\Omega_N} \Phi(a)\mathrm{d} \tilde{m}_\omega$.
\end{enumerate}
\end{theorem}

\begin{remark}
The importance of part 3 of Theorem \ref{mainthm} is in the context of the states constructed from $\theta$-summable spectral triples in \cite{froehlicetal}. The assumption \cite[Assumption 5.4]{froehlicetal} requires the associated states to be tracial. This condition clearly fails in the purely infinite case.
\end{remark}

A key ingredient in the proof of the theorem is the notion of the depth-kore operator from \cite{GMR}. The depth-kore operator $\kappa$ is a self-adjoint operator on $L^2(O_N)$ which together with $c$ facilitates a decomposition $L^2(O_N)=\bigoplus_{n,k}\mathcal{H}_{n,k}$ into finite-dimensional subspaces with an explicit ON-basis. As we will see below in Proposition \ref{prop:Fock} of Section \ref{subsec:2.2}, $P_\mathcal{F}$ is the orthogonal projection onto the free Fock space $\mathcal{F}:= \bigoplus_{n=0}^\infty\mathcal{H}_{n,0}\cong \ell^2(\mathcal{V}_N)$ where $\mathcal{V}_N=\cup_{k=0}^\infty \{1,\ldots, N\}^k$.

The structure of the paper is as follows. In Section \ref{sec:2} we describe the geometry of the Cuntz-Renault groupoid $\mathcal{G}_{N}$ and the GNS representation of the KMS state of the Cuntz algebra in terms of the Cuntz-Renault groupoid. We compare the $\kappa$-function on $\mathcal{G}_{N}$ (cf. \cite[Section 5]{GM}) to the $\kappa$-operator on its $L^2$-space (cf. \cite[Lemma 2.13]{GMR}) in Section \ref{sec:3}. The integral operator $T$ is computed in Section \ref{sec:4} and we assemble these ingredients to spectral triples in Section \ref{sec:5}. The proof of Theorem \ref{mainthm} is found in Section \ref{sec:5} and \ref{subsec:5.1}. 

\section{Metric measure theory on $O_N$}
\label{sec:2}
In this section we will set the scene for the paper and describe the relevant objects. Most of this material reviews previously published results. The context we present, which to our knowledge is novel, sheds a new light on them.

\subsection{The groupoid $\mathcal{G}_N$ as a metric measure space}

\label{subsec:2.1}

The groupoid $\mathcal{G}_N$ was defined as a set with algebraic structure in Equation \eqref{defininggn} and we now describe its topology in more detail. Define the functions $\kappa_\mathcal{G}:\mathcal{G}_N\to \mathbb{N}$ and $c:\mathcal{G}_N\to \mathbb{Z}$ by
$$\kappa_\mathcal{G}:(x,n,y)\mapsto \min\{k\geq \max\{0,-n\}: \sigma^{n+k}(x)=\sigma^k(y)\},\quad c: (x,n,y)\mapsto n.$$
For $g\in\mathcal{G}_{N}$ and composable $g_{1},g_{2}\in\mathcal{G}_{N}$ it holds that 
$$c(g_{1}\cdot g_{2})=c(g_{1})+c(g_{2}),\quad \kappa_{\mathcal{G}}(g_{1}\cdot g_{2})\leq \kappa(g_{1})+\kappa(g_{2}),\quad c(g)+\kappa_\mathcal{G}(g)\geq 0.$$
In summary, $c$ is a cocycle, $\kappa_{\mathcal{G}}$ is submultiplicative and their sum is a positive function. 

We equip $\mathcal{G}_N$ with the smallest topology making $c$, $\kappa_\mathcal{G}$, $r_\mathcal{G}$ and $d_\mathcal{G}$ continuous. It is readily verified that a basis for the topology on $\mathcal{G}_N$ is given by the sets
$$X_{\mu,\nu}:=\{(x,|\mu|-|\nu|,y)\in \mathcal{G}_N: x\in C_\mu, \; y\in C_\nu, \; \sigma^{|\mu|}(x)=\sigma^{|\nu|}(y)\}, \quad\mbox{for  } \mu,\nu\in \mathcal{V}_N.$$
The groupoid $\mathcal{G}_N$ is \'{e}tale in this topology. An \'{e}tale groupoid over a totally disconnected space is again totally disconnected, so the space of compactly supported locally constant functions is dense in $C_c(\mathcal{G}_N)$. For $\mu,\nu \in \mathcal{V}_N$ we use the notation $S_\mu:=S_{\mu_1}\cdots S_{\mu_{|\mu|}}$, $S_\mu^*:=(S_\mu)^*$ and $\chi_{\mu,\nu}$ for the characteristic function of $X_{\mu,\nu}$. The following result is proven in \cite{Ren1, Ren2}.

\begin{theorem}
\label{theisoof}
The $C^*$-algebras $O_N$ and $C^*(\mathcal{G}_N)$ are isomorphic via a $*$-homomorphism $O_N\to C^*(\mathcal{G}_N)$ that maps $S_\mu S_\nu^*$ to the compactly supported locally constant function $\chi_{\mu,\nu}\in C_c(\mathcal{G}_N)$.
\end{theorem}

\begin{notn}
For $g=(x,n,y)\in \mathcal{G}_N$ we have $\sigma^{n+\kappa_\mathcal{G}(g)}(x)=\sigma^{\kappa_\mathcal{G}(g)}(y)$. We will use the notation $z(g):= \sigma^{\kappa_\mathcal{G}(g)}(y)$, $\mu_\mathcal{G}(g)$ will denote the word of length $n+\kappa_\mathcal{G}(g)$ such that $x=\mu_\mathcal{G}(g)z(g)$ and  $\nu_\mathcal{G}(g)$ will denote the word of length $\kappa_\mathcal{G}(g)$ such that $y=\nu_\mathcal{G}(g)z(g)$. In particular, we have
$$g=(\mu_\mathcal{G}(g)z(g),c(g),\nu_\mathcal{G}(g)z(g)), \quad\forall g\in \mathcal{G}_N.$$
Clearly, $z:\mathcal{G}_N\to \Omega_N$ and $\mu_\mathcal{G},\nu_\mathcal{G}:\mathcal{G}_N\to \mathcal{V}_N$ are continuous. When there is no risk of confusion with fixed finite words, we write simply $\mu(g)$ and $\nu(g)$. We also write $y(g):=y$.
\end{notn}

The compact space $\Omega_N$ is metrized by the metric $\rho_\Omega$ defined by 
$$\rho_\Omega(x_1x_2\cdots ,y_1y_2\cdots ):=\inf\{\mathrm{e}^{-l}: x_1x_2\cdots x_l= y_1y_2\cdots y_l\},$$
with the convention that $\rho_\Omega(ix_2\cdots ,jy_2\cdots )=1$ if $i\neq j$.

\begin{definition}
\label{definingmetric}
We define $\rho_\mathcal{G}:\mathcal{G}_N\times \mathcal{G}_N\to [0,\infty]$ by 
$$\rho_\mathcal{G}(g_1,g_2):=
\begin{cases}
\infty, \;&\mbox{  if $\kappa_\mathcal{G}(g_1)\neq \kappa_\mathcal{G}(g_2)$ or $\mu(g_1)\neq \mu(g_2)$,}\\
\rho_\Omega(y(g_1),y(g_2)), \;&\mbox{  if $\kappa_\mathcal{G}(g_1)= \kappa_\mathcal{G}(g_2)$ and $\mu(g_1)= \mu(g_2)$.}
\end{cases}$$
\end{definition}

For $\mu\in \mathcal{V}_N$ and $k\in \mathbb{N}$, we define the set 
\begin{equation}
\label{thesetcmuk}
\mathcal{C}_{\mu,k}:=\{g\in\mathcal{G}_N:\mu_\mathcal{G}(g)=\mu, \kappa_\mathcal{G}(g)=k\}.
\end{equation}
The set $\mathcal{C}_{\mu,k}$ is homeomorphic to a clopen subset of $\Omega_N$ via the domain mapping $d_\mathcal{G}$. We can clearly partition 
$$\mathcal{G}_N=\dot{\cup}_{\mu, k} \mathcal{C}_{\mu,k}.$$ 
Moreover, for a fixed $g_1\in \mathcal{G}_N$ we have
$$\{g_2\in \mathcal{G}_N:\rho_\mathcal{G}(g_1,g_2)<\infty\}=C_{\mu(g_1),\kappa(g_1)}.$$

\begin{proposition}
\label{propertiesofrhog}
The function $\rho_\mathcal{G}$ is an extended metric on $\mathcal{G}_N$ and the topology induced by $\rho_{\mathcal{G}}$ coincides with the \'{e}tale topology on $\mathcal{G}_{N}$. The functions $c$, $\kappa_\mathcal{G}$, $r_\mathcal{G}$ and $d_\mathcal{G}$ as well as any compactly supported locally constant function are uniformly Lipschitz continuous with respect to $\rho_{\mathcal{G}}$.  
\end{proposition}

\begin{proof}
We start by giving the argument for why $\rho_\mathcal{G}$ is an extended metric. If $\rho_\mathcal{G}(g_1,g_2)=0$ then $y(g_1)=y(g_2)$, $\kappa_\mathcal{G}(g_1)= \kappa_\mathcal{G}(g_2)$ and $\mu_\mathcal{G}(g_1)= \mu_\mathcal{G}(g_2)$ so $c(g_1)=c(g_2)$ and we conclude that $g_1=g_2$. The function $\rho_\mathcal{G}$ is clearly non-negative and symmetric. The triangle inequality follows from the fact that given $\mu\in \mathcal{V}_N$ and $k\in \mathbb{N}$, the set $\mathcal{C}_{\mu,k}$ is bi-Lipschitz homeomorphic to a clopen subset of $\Omega_N$ via the domain mapping $d_\mathcal{G}$. The remainder of the proposition are direct consequences of the construction of the extended metric.
\end{proof}

\begin{remark}
We note that $r_\mathcal{G}$ and $d_\mathcal{G}$ are locally bi-Lipschitz homeomorphisms between $\mathcal{G}_N$ and $\Omega_N$ so any local metric invariant, e.g. Hausdorff dimension, remains the same for the two spaces. 
\end{remark}

It is often fruitful to think of $\Omega_N$ as the Gromov boundary of the discrete hyperbolic space $\mathcal{V}_N$. Here we think of $\mathcal{V}_N$ as a rooted tree, with root $\emptyset\in \{1,\ldots, N\}^0=\{\emptyset\}$ and given a directed graph structure by declaring an edge from $\mu$ to $\mu j$ for any $\mu \in \mathcal{V}_N$ and $j\in \{1,\ldots,N\}$. We write $\overline{\mathcal{V}}_N:=\mathcal{V}_N\cup \Omega_N$ for the corresponding compactification of $\mathcal{V}_N$; we topologize $\overline{\mathcal{V}}_N$ in such a way that $\mathcal{V}_N\subseteq \overline{\mathcal{V}}_N$ is a discrete subspace and for any $\mu \in \mathcal{V}_N$, the set $\{\nu\in \mathcal{V}_N: \nu=\mu\nu_0$ for some $\nu_0\in \mathcal{V}_N\}\cup C_\mu$ is open. Let $\delta_{\mu}$ denote the Dirac measure at $\mu\in\mathcal{V}_{N}$ and for $s>\log(N)$ define probability measures on $\overline{\mathcal{V}}_{N}$ via
$$m_s:=\frac{\sum_{\mu\in \mathcal{V}_N}\mathrm{e}^{-s|\mu|}\delta_\mu}{\sum_{\mu\in \mathcal{V}_N}\mathrm{e}^{-s|\mu|}}.$$
The measures $m_s$ are supported in $\mathcal{V}_N$. The following construction of the Patterson-Sullivan measures on $\Omega_N$ is well-known, see for instance \cite{Coornaert}.

\begin{proposition}
The net of measures $(m_s)_{s>\log(N)}$ has a w$^*$-limit $m_\Omega$ as $s\to \log(N)$. The measure $m_{\Omega}$ is supported on $\Omega_{N}\subset\overline{\mathcal{V}}_{N}$ and coincides with $\log (N)$-dimensional Hausdorff measure. It satisfies $m_\Omega(C_\nu)=N^{-|\nu|}$ for any $\nu\in \mathcal{V}_N$.  
\end{proposition}

\subsection{The representation associated with the KMS state on $O_N$}
\label{subsec:2.2}

We will now approach $O_N$ from an operator theoretic viewpoint. The cocycle $c$ gives rise to a $U(1)$-action on $O_N$ by $(z\cdot f)(g)=z^{c(g)}f(g)$ for $f\in C_c(\mathcal{G}_N)$. Under the isomorphism of Theorem \ref{theisoof}, this action is given on the generators of $O_N$ by $z\cdot S_i=zS_i$. The functional 
$$\phi(f):=\int_{\Omega_N} f(x,0,x)\mathrm{d}m_\Omega,$$
extends to a state on $O_N$. Indeed, $\phi(S_\mu S_\nu^*)=\delta_{\mu,\nu} N^{-|\nu|}$. The state $\phi$ is the unique KMS state on $O_N$ (equipped with the action defined above) and its inverse temperature is $\log(N)$, see \cite{OP}.

\begin{proposition}
We consider the measure $m_\mathcal{G}:=d_\mathcal{G}^* m_\Omega$ on $\mathcal{G}_N$. The isomorphism of Theorem \ref{theisoof} uniquely determines a unitary isomorphism $L^2(O_N,\phi)\to L^2(\mathcal{G}_N,m_\mathcal{G})$ compatible with the left $O_N$-action.
\end{proposition}

The proposition follows using the fact that $1\in O_N$, which corresponds to $\chi_\Omega\in C_c(\mathcal{G}_N)\subseteq C^*(\mathcal{G}_N)$, satisfies $\phi(f)=\langle 1, f*1\rangle_{L^2(\mathcal{G}_N,m_{\mathcal{G}})}$. Motivated by this result, we identify $L^2(O_N,\phi)$ with $L^2(\mathcal{G}_N,m_\mathcal{G})$ and write simply $L^2(O_N)$.

\begin{definition}
\label{defofbasis}
For a finite word $\mu\in \mathcal{V}_N$ we write $t(\mu):=\mu_{|\mu|}$ for $\mu$ non-empty and $t(\emptyset)=\emptyset$. We define $(\mathrm{e}_{\mu,\nu})_{\mu,\nu \in \mathcal{V}_N}\subseteq L^2(O_N)$ by $\mathrm{e}_{\emptyset,\emptyset}=\chi_{\Omega}$ and
$$\mathrm{e}_{\mu,\nu}:=
\begin{cases}
N^{|\nu|/2} S_{\mu}S_\nu^*, \; &t(\mu)\neq t(\nu),\\
\\
N^{|\nu|/2}\sqrt{\frac{N}{N-1}} \left(S_{\mu}S_\nu^*-N^{-1}S_{\underline{\mu}}S_{\underline{\nu}}^*\right), \; &t(\mu)= t(\nu)\neq \emptyset.
\end{cases}$$
Here we have written $\mu=\underline{\mu}t(\mu)$ and $\nu=\underline{\nu}t(\nu)$.
\end{definition}

\begin{proposition}[Lemma 2.13 of \cite{GMR}]
The collection $(\mathrm{e}_{\mu,\nu})_{\mu,\nu \in \mathcal{V}_N}\subseteq L^2(O_N)$ is an ON-basis. 
\end{proposition}

\begin{definition}
Following \cite{GMR}, we define the depth-kore operator $\kappa$ on $L^2(O_N)$ as the densely defined self-adjoint operator such that 
$$\kappa \mathrm{e}_{\mu,\nu}=|\nu|\mathrm{e}_{\mu,\nu}.$$ 
We define the operator $c$ on $L^2(O_N)$ as the densely defined self-adjoint operator such that 
$$c \mathrm{e}_{\mu,\nu}=(|\mu|-|\nu|)\mathrm{e}_{\mu,\nu}.$$
\end{definition}

We note that by construction, $c$ commutes with $\kappa$ on a common core and $c+\kappa$ is positive. We define the $N^{n+2k}$-dimensional space
\begin{equation}
\label{herewedefine}
\mathcal{H}_{n,k}:=\ker(c-n)\cap \ker(\kappa-k)=l.s.\{\mathrm{e}_{\mu,\nu}: |\mu|=n+k, |\nu|=k\}.
\end{equation}

\begin{proposition}\label{prop:Fock}
Let $X_\mathcal{F}$ denote the set from Equation \eqref{thesetxf} and $P_\mathcal{F}$ the integral operator with kernel $\chi_{X_\mathcal{F}}$. The operator $P_\mathcal{F}$ is the orthogonal projection onto the Fock space $\mathcal{F}:=\ker \kappa =\bigoplus_{n=0}^\infty \mathcal{H}_{n,0}$. Moreover, $P_\mathcal{F}$ preserves the domain of $c$ and $\kappa$ and commutes with $c$ and $\kappa$ on their respective domains.
\end{proposition}

\begin{proof}
The integral kernel of the orthogonal projection onto the Fock space $\mathcal{F}$ is given by the function 
$$\sum_{n=0}^\infty \sum_{|\mu|=n} \mathrm{e}_{\mu,\emptyset}(g_1)\mathrm{e}_{\mu,\emptyset}(g_2)=\sum_{n=0}^\infty \sum_{|\mu|=n}\chi_{X_{\mu,\emptyset}\times X_{\mu,\emptyset}}(g_1,g_2)=\chi_{\cup_\mu X_{\mu,\emptyset}\times X_{\mu,\emptyset}}(g_1,g_2).$$
The proposition follows from the fact that $X_\mathcal{F}=\bigcup_\mu \left(X_{\mu,\emptyset}\times X_{\mu,\emptyset}\right)$.
\end{proof}

\begin{remark}
The isometry $v:\ell^2(\mathcal{V}_N)\to L^2(O_N)$, $\delta_\mu\mapsto \mathrm{e}_{\mu,\emptyset}$ surjects onto the Fock space $\mathcal{F}=\ker(\kappa)$ (compare \cite[Remark 2.2.4]{GM}). We also note that there is an isometry $L^2(\Omega_N,m_\Omega) \to L^2(O_N)$ mapping surjectively onto the ``anti-Fock space" $\mathcal{F}^{an}:=\ker(c+\kappa)$ via $\chi_{C_\mu}\mapsto S_\mu S_\mu^*$. The ``anti-Fock space" is often a source of trouble, see \cite[Proof of Theorem 2.19]{GMR}. The basis $(e_{\mu,\mu})_{\mu\in \mathcal{V}_N}$ for $L^2(\Omega_N,m_\Omega)\subseteq L^2(O_N)$ is related to the wavelet basis studied in \cite{gillaspyetal}.
\end{remark}

In \cite{GMR} the operators $c$, $\kappa$ and $P_\mathcal{F}$ were assembled into a spectral triple. We define $D_\kappa$ as the closure of $(2P_{\mathcal{F}}-1)|c|-\kappa$. It was proven in \cite{GMR} that $(O_N,L^2(O_N),D_\kappa)$ is a spectral triple whose class coincides with $\widehat{[1]}\in K^1(O_N)$. The explicit construction is motivated by the $K$-homological information carried by the projection $P_\mathcal{F}\equiv \chi_{[0,\infty)}(D_\kappa)$. The aim of this paper is to give a more geometric construction of a spectral triple on $O_N$, with the same $K$-homological content.

\section{The $\kappa$-function and $\kappa$-operator on $O_N$}
\label{sec:3}

An important aspect in the noncommutative geometry of the Cuntz algebra $O_N$ is the distinction between the depth-kore function $\kappa_\mathcal{G}$ in the groupoid $\mathcal{G}_N$ and the depth-kore operator $\kappa$. Both are invariants of Cuntz-Pimsner constructions of $O_N$: the depth-kore function from $O_N$ as a Cuntz-Pimsner algebra with coefficients $C(\Omega_N)$ and the depth-kore operator from $O_N$ as a Cuntz-Pimsner algebra with coefficients $\mathbb{C}$. These two models are discussed in \cite{GMR}, notably in \cite[Section 2.5.3]{GMR}. 

Let us go into the details of the other approach using $C(\Omega_N)$ as coefficients. The details can be found in \cite{GM, GMR}. Let $\Xi_N$ denote the $C(\Omega_N)$-Hilbert $C^*$-module completion of $C_c(\mathcal{G}_N)$ in the $C(\Omega_N)$-valued inner product $\langle f_1,f_2\rangle_{C(\Omega)}:=\Phi(f_1^**f_2)$ where $\Phi$ denotes the conditional expectation $C_c(\mathcal{G}_N)\to C(\Omega_N)$ onto the unit space obtained from the inclusion \eqref{unitincl}. Multiplication by the functions $c$ and $\kappa_\mathcal{G}$ define self-adjoint regular operators on $\Xi_N$. The Fock module $\mathcal{F}_\Omega:=\ker \kappa_\mathcal{G}\subseteq \Xi_N$ is complemented and the adjointable projection
$$P_\Omega=\chi_{\{0\}}(\kappa_\mathcal{G}): \Xi_{N}\to \Xi_{N},$$
satisfies $\mathcal{F}_{\Omega}=P_{\Omega}\Xi_{N}$. Following the recipe above, we define a self-adjoint regular operator $D_\Omega$ on $\Xi_N$ as the closure of $(2P_{\Omega}-1)|c|-\kappa_\mathcal{G}$. Then $P_\Omega=\chi_{[0,\infty)}(D_\Omega)$ and this operator projects onto the $C(\Omega)$-Hilbert $C^*$-submodule spanned by $\{S_\mu: \mu\in \mathcal{V}_N\}$. The triple $(O_N, \Xi_N,D_\Omega)$ defines an unbounded $(O_N,C(\Omega_N))$-Kasparov module. In this instance, $(O_N, \Xi_N,D_\Omega)$ can be thought of as a bundle of spectral triples over $\Omega_N$. We recall the following result from \cite[Theorem 5.2.3]{GM}.

\begin{theorem}
Let $w\in \Omega_N$ and denote the discrete $d_\mathcal{G}$-fiber by $\mathcal{V}_w:=d_\mathcal{G}^{-1}(w)\subseteq \mathcal{G}_N$. The $C^*$-algebra $O_N$ acts on $\ell^2(\mathcal{V}_w)$ via the groupoid structure. Define $D_w$ as a self-adjoint operator on $\ell^2(\mathcal{V}_w)$ by 
$$D_wf(x,n,w):=|n|(2P_w-1)f(x,n,w)-\kappa_\mathcal{G}(x,n,w)f(x,n,w),$$
where $P_w$ denotes the projection onto the closed linear span of the orthogonal set $\{\chi_{X_{\mu,\emptyset}}|_{\mathcal{V}_w}: \mu\in \mathcal{V}_N\}\subseteq \ell^2(\mathcal{V}_N)$. Then $(O_N,\ell^2(\mathcal{V}_N), D_w)$ is a $\theta$-summable spectral triple whose phase $D_w|D_w|^{-1}$ defines a finitely summable Fredholm module representing the class $\widehat{[1]}\in K^1(O_N)$. 
\end{theorem}

A key step in proving that $(O_N, \Xi_N,D_\Omega)$ is a Kasparov module is the study of the submodules $\Xi_{n,k}:=\ker(c-n)\cap \ker(\kappa_\mathcal{G}-k)\subseteq \Xi_N$. The modules do in this instance carry geometric content as
$$\Xi_{n,k}=C(\mathcal{G}_{n,k}),\quad \mbox{where}\quad \mathcal{G}_{n,k}=c^{-1}(\{n\})\cap \kappa_\mathcal{G}^{-1}(\{k\})\subseteq \mathcal{G}_N.$$
The set $\mathcal{G}_{n,k}$ is compact and $C(\mathcal{G}_{n,k})$ is a finitely generated projective $C(\Omega_N)$-module. Later in the paper, we will need to make use of the interaction between the depth-kore function $\kappa_{\mathcal{G}}$ and the depth-kore operator $\kappa$. 

\begin{definition}
For two finite words $\mu,\nu\in \mathcal{V}_N$ we write $\mu\wedge \nu$ for the longest word such that $\mu=\mu_0(\mu\wedge \nu)$ and $\nu=\nu_0(\mu\wedge \nu)$ for some words $\mu_0$ and $\nu_0$. We define 
$$\kappa_\mathcal{V}(\mu,\nu):=|\nu|-|\mu\wedge \nu|.$$
\end{definition}

\begin{proposition}
\label{kappagonbasis}
For $\mu,\nu\in \mathcal{V}_N$, $\kappa_\mathcal{G} \mathrm{e}_{\mu,\nu}=\kappa_\mathcal{V}(\mu,\nu)\mathrm{e}_{\mu,\nu}$. In particular, $\mathrm{e}_{\mu,\nu}\in \Xi_{n,k}$ if and only if $n=|\mu|-|\nu|$ and $k=\kappa_\mathcal{V}(\mu,\nu)$.
\end{proposition}

The proof consists of a long inspection to verify that $\mathrm{supp}( \mathrm{e}_{\mu,\nu})\cap \mathcal{G}_{n,k}\neq \emptyset$ if and only if $n=|\mu|-|\nu|$ and $k=\kappa_\mathcal{V}(\mu,\nu)$. A key point in the proof, putting the two cases in Definition \ref{defofbasis} on equal footing, is the identity: 
$$\kappa_\mathcal{V}(\mu i,\nu j)=\delta_{i,j}\kappa_\mathcal{V}(\mu ,\nu )+(1-\delta_{i,j})(|\nu|+1).$$
Using the fact that $0\leq \kappa_{\mathcal{V}}(\mu,\nu)\leq |\nu|$ we deduce the next Corollary from Proposition \ref{kappagonbasis}.

\begin{corollary}
\label{kapakapa}
As self-adjoint operators on $L^2(O_N)$, the operator $\kappa_\mathcal{G}$ is relatively bounded by $\kappa$ with relative norm bound $1$. Moreover, $\kappa_\mathcal{G}$ and $\kappa$ commute on a common core.
\end{corollary}

\section{An integral operator on $O_N$}
\label{sec:4}

In this section, we define the singular integral operator that is used to construct spectral triples on $O_N$. The singular integral operator will at large behave like the depth-kore operator $\kappa$. 

\begin{definition}
Define $C^\infty_c(\mathcal{G}_N)\subseteq C_c(\mathcal{G}_N)$ as the subspace of all compactly supported locally constant functions. Define the operator $T_0: C^\infty_c(\mathcal{G}_N)\to L^2(\mathcal{G}_N)$ by 
$$T_0f(g_1):=\frac{1}{1-N^{-1}} \int_{\mathcal{G}_N} \frac{f(g_1)-f(g_2)}{\rho_\mathcal{G}(g_1,g_2)^{\log(N)}}\mathrm{d}m_\mathcal{G}(g_2).$$
In the integrand, we apply the convention that $\frac{c}{\infty}=0$ for any finite number $c$.
\end{definition}

We will compute $T_0$ in the basis $\mathrm{e}_{\mu,\nu}$ of $L^2(O_N)$ and since $C_{c}^{\infty}(\mathcal{G}_{N})=\mbox{span}\{e_{\mu,\nu}\}$. The computation shows that $T_{0}$ is well-defined and maps $C_{c}^{\infty}(\mathcal{G}_N)$ into $L^2(O_N)$. More precisely, the computation shows that $T_0$ is up to a bounded operator diagonal in the basis $\mathrm{e}_{\mu,\nu}$, and as such we can extend $T_0$ to a densely defined self-adjoint operator 
$$T:\mbox{Dom } T\subset L^2(O_N)\to L^{2}(O_{N}),$$ 
with $C_{c}^{\infty}(\mathcal{G}_{N})\subseteq \mbox{Dom }T$. 

First we define the so called dispersion operator. For two words $\mu,\nu\in \mathcal{V}_N$ we write $\mu\vee \nu$ for the finite word of maximal length such that $\mu=(\mu\vee \nu)\mu_0$ and $\nu=(\mu\vee \nu)\nu_0$ for some finite words $\mu_0,\nu_0\in \mathcal{V}_N$.

\begin{lemma}
\label{displem}
Define the dispersion operator $B$ on $L^2(O_N)$ by the formula
$$B\mathrm{e}_{\mu,\nu}=(1-N^{-1})^{-1}(1-\delta_{t(\mu),t(\nu)})\sum_{m\neq t(\mu)}\sum_{\ell=0}^{|\nu|-1}\sum_{|\gamma|=|\nu|-1, \, |\gamma\vee \nu|=\ell} N^{\ell-|\nu|} \mathrm{e}_{\mu, \gamma m}.$$
Then the operator $B$ is a well defined bounded self-adjoint operator commuting with $\kappa$, $c$ and $\kappa_{\mathcal{G}}$ on a common core. In fact $\kappa|_{(\ker B)^\perp}=\kappa_{\mathcal{G}}|_{(\ker B)^\perp}$.
\end{lemma}

\begin{proof}
The operator $B$ is defined on an ON-basis and it is clear from the expression that $B$ is self-adjoint if $B$ is bounded. The only non-trivial fact to prove is therefore that $B$ is bounded. We compute that 
\begin{align*}
\|B\mathrm{e}_{\mu,\nu}\|^2&=(1-N^{-1})^{-2}(1-\delta_{t(\mu),t(\nu)})\sum_{m\neq t(\mu)}\sum_{\ell=0}^{|\nu|-1}\sum_{|\gamma|=|\nu|-1, \, |\gamma\vee \nu|=\ell} N^{2\ell-2|\nu|}\\
&= (1-\delta_{t(\mu),t(\nu)})\frac{N}{N-1}\sum_{\ell=0}^{|\nu|-1} N^{\ell-|\nu|} =(1-\delta_{t(\mu),t(\nu)})\frac{N}{(N-1)^2}(1-N^{-|\nu|}). 
\end{align*}
It follows that $B$ is bounded. Since $(\ker B)^\perp$ is spanned by basis vectors $\mathrm{e}_{\mu,\nu}$ where $t(\mu)\neq t(\nu)$ and $\kappa_\mathcal{V}(\mu,\nu)=|\nu|$ if $t(\mu)\neq t(\nu)$, Proposition \ref{kappagonbasis} implies $\kappa|_{(\ker B)^\perp}=\kappa_{\mathcal{G}}|_{(\ker B)^\perp}$.
\end{proof}

We define $Q$ as the orthogonal projection onto the closed linear span of the set $\{\mathrm{e}_{\mu,\nu}: 0\leq \kappa_\mathcal{V}(\mu,\nu)<|\nu|\}$. That is, $\mathrm{e}_{\mu,\nu}\in QL^2(O_N)$ if and only if $t(\mu)=t(\nu)\neq \emptyset$.

\begin{theorem}
\label{bigtcomp}
As an operator on $C^\infty_c(\mathcal{G}_N)$, we have that 
$$T_0=\kappa-\frac{1}{N}\kappa_\mathcal{G}+(N-1)^{-1}Q-B.$$ 
In particular, $T_0$ is a well-defined operator with dense domain and extends to a self-adjoint operator $T$ on $L^2(O_N)$ with discrete spectrum and $T-\kappa$ is relatively bounded by $\kappa$ and $T$ with relative norm bound $1/N$. 
\end{theorem}

The proof of this theorem will occupy the rest of this section. To prove the theorem, it suffices to prove that it holds when acting on basis elements $\mathrm{e}_{\mu,\nu}$: they span the compactly supported locally constant functions. For $g_1\in \mathcal{G}_N$ and $\ell\in \mathbb{N}$, we introduce the notation 
$$X^\ell(g_1):=\{g_2\in \mathcal{G}_N: \rho_\mathcal{G}(g_1,g_2)=\mathrm{e}^{-\ell}\}.$$
Note that $(X^\ell(g_1))_{\ell\in \mathbb{N}}$ is a clopen partition of $\mathcal{C}_{\mu(g_1),\kappa_\mathcal{G}(g_1)}$. In fact, for any $g_1\in \mathcal{G}_N$, we can make a disjoint clopen partition $\mathcal{G}_N=\{g_2: \rho_\mathcal{G}(g_1,g_2)=\infty\}\cup (\cup_{\ell=0}^\infty X^\ell(g_1))$. From this discussion, it follows that 
\begin{equation}
\label{decomposingt}
T=\frac{1}{1-N^{-1}} \sum_{\ell=0}^\infty N^\ell T_\ell, \quad\mbox{where   } T_\ell f(g_1):=\int_{X^\ell(g_1)} (f(g_1)-f(g_2))\mathrm{d}m_\mathcal{G}(g_2).
\end{equation}
To compute $T$ in the ON-basis, we first compute $T_\ell$ on the characteristic functions $\chi_{X_{\mu,\nu}}$. To ease notation, we write $\chi_{\mu,\nu}=\chi_{X_{\mu,\nu}}\in C_c(\mathcal{G}_N)$ and $\chi_\nu=\chi_{C_\nu}\in C(\Omega_N)$.  We have that 
\begin{align*}
T_\ell\chi_{\mu,\nu}(g_1)=m_\mathcal{G}&(X^\ell(g_1)\cap X_{\mu,\nu})(\chi_{\mu,\nu}(g_1)-1)\\
&+m_\mathcal{G}(X^\ell(g_1)\setminus (X^\ell(g_1)\cap X_{\mu,\nu}))\chi_{\mu,\nu}(g_1).
\end{align*}
We now proceed to compute the relevant volumes appearing in this expression. 

\begin{lemma}
\label{lemmaa}
For $g_1\notin X_{\mu,\nu}$,
$$m_\mathcal{G}(X^\ell(g_1)\cap X_{\mu,\nu})=N^{-|\nu|} \sum_{|\gamma|=\kappa_\mathcal{V}(\mu,\nu)=\kappa_\mathcal{V}(\mu,\gamma)} \chi_{\underline{\mu}_{n+|\gamma|},\gamma}(g_1)(\chi_{\underline{\nu}_{\ell}}(y_1)-\chi_{\underline{\nu}_{\ell+1}}(y_1)),$$
where $\underline{\nu}_{\ell}$ and $\underline{\nu}_{\ell+1}$ denotes the first $\ell$ and $\ell+1$ letters of $\nu$ and $\underline{\mu}_{n+|\gamma|}$ the first $n+|\gamma|=n+\kappa_\mathcal{V}(\mu,\nu)$ letters of $\mu$. We interpret $\underline{\nu}_\ell=\nu$ if $\ell\geq |\nu|$.
\end{lemma}

\begin{proof}
Take $g_1\notin X_{\mu,\nu}$. Firstly, suppose that $g_1\in \cup_{|\gamma|=\kappa_\mathcal{V}(\mu,\nu)=\kappa_\mathcal{V}(\mu,\gamma)} X_{\underline{\mu}_{n+|\gamma|},\gamma}$ and that $y_1\in C_{\underline{\nu}_{\ell}}\setminus C_{\underline{\nu}_{\ell+1}}$. This means precisely that the domain mapping defines a measure preserving bi-Lipschitz homeomorphism $X^\ell(g_1)\cap X_{\mu,\nu}\to C_\nu$. Conversely, if $g_1\notin \cup_{|\gamma|=\kappa_\mathcal{V}(\mu,\nu)=\kappa_\mathcal{V}(\mu,\gamma)} X_{\underline{\mu}_{n+|\gamma|},\gamma}$ or $y_1\notin C_{\underline{\nu}_{\ell}}\setminus C_{\underline{\nu}_{\ell+1}}$ then $X^\ell(g_1)\cap X_{\mu,\nu}$ is empty, so $X^\ell(g_1)\cap X_{\mu,\nu}$ has measure zero.
\end{proof}

\begin{lemma}
\label{lemmab}
For $g_1\in X_{\mu,\nu}$, 
$$m_\mathcal{G}(X^\ell(g_1)\setminus (X^\ell(g_1)\cap X_{\mu,\nu}))=
\begin{cases} 
0, \; &\ell\geq |\nu|,\\
N^{-\ell}-N^{-\ell-1}, \; &\kappa_\mathcal{G}(g_1)\leq \ell< |\nu|,\\
N^{-\ell-1}(N-1)^2, \; &0\leq \ell< \kappa_\mathcal{G}(g_1).\end{cases}$$
\end{lemma}

\begin{proof}
Take $g_1\notin X_{\mu,\nu}$. A computation with cylinder sets shows that 
\begin{align*}
m_\mathcal{G}(X^\ell(g_1)\cap X_{\mu,\nu})&=
\begin{cases} 
0, \; &\ell< |\nu|,\\
N^{-\ell}-N^{-\ell-1}, \; & \ell\geq |\nu|,\end{cases}\\
\mbox{and}\quad m_\mathcal{G}(X^\ell(g_1))&=
\begin{cases} 
N^{-\ell}-N^{-\ell-1}, \; &\kappa_\mathcal{G}(g_1)\leq \ell,\\
N^{-\ell-1}(N-1)^2, \; &0\leq \ell< \kappa_\mathcal{G}(g_1).\end{cases}
\end{align*}
The result follows by subtracting the two expressions.
\end{proof}

From these computations, we deduce a simple special case of Theorem \ref{bigtcomp}. A short computation shows that $T_\ell \chi_{\mu,\emptyset}=0$ for any finite word $\mu$ and $\ell\in \mathbb{N}$. Therefore
$$T\mathrm{e}_{\mu,\emptyset}=0. $$
We now turn to the general case.

\begin{proof}[Proof of Theorem \ref{bigtcomp}]
Using the decomposition \eqref{decomposingt}, Lemma \ref{lemmaa} and Lemma \ref{lemmab} we write 
\begin{align}
\nonumber
(1-N^{-1})T\chi_{\mu,\nu}=&\sum_{\ell=0}^{|\nu|-1}  \sum_{|\gamma|=\kappa_\mathcal{V}(\mu,\nu)=\kappa_\mathcal{V}(\mu,\gamma)}N^{\ell-|\nu|} \chi_{\underline{\mu}_{n+|\gamma|},\gamma}(g_1)(\chi_{\underline{\nu}_{\ell+1}}(y_1)-\chi_{\underline{\nu}_{\ell}}(y_1))\\
\nonumber
&+\left(\sum_{\ell=0}^{\kappa_\mathcal{G}(g_1)-1} N^{-2}(N-1)^2+\sum_{\ell=\kappa_\mathcal{G}(g_1)}^{|\nu|-1}(1-N^{-1})\right)\chi_{\mu,\nu}(g_1)\\
\nonumber
=&\sum_{\ell=0}^{|\nu|-1}  \sum_{|\gamma|=\kappa_\mathcal{V}(\mu,\nu)=\kappa_\mathcal{V}(\mu,\gamma)}N^{\ell-|\nu|} \chi_{\underline{\mu}_{n+|\gamma|},\gamma}(g_1)(\chi_{\underline{\nu}_{\ell+1}}(y_1)-\chi_{\underline{\nu}_{\ell}}(y_1))\\
\label{tchimunu}
&+(1-N^{-1})\left(-\frac{\kappa_\mathcal{G}(g_1)}{N}+|\nu|\right)\chi_{\mu,\nu}(g_1).
\end{align}

Take two words $\mu,\nu\in \mathcal{V}_N$ with $t(\mu)=t(\nu)\neq \emptyset$. Up to normalization, $\mathrm{e}_{\mu,\nu}$ coincides with $\chi_{\mu,\nu}-N^{-1}\chi_{\underline{\mu},\underline{\nu}}$. Here we are using the notation of Definition \ref{defofbasis}, and hence $|\underline{\nu}|=|\nu|-1$ whenever $|\nu|>0$. We compute that 
\begin{align*}
(1-N^{-1})&T(\chi_{\mu,\nu}-N^{-1}\chi_{\underline{\mu},\underline{\nu}})\\
=&(1-N^{-1})\left(-\frac{\kappa_\mathcal{G}(g_1)}{N}+|\nu|\right)\chi_{\mu,\nu}(g_1)\\
&-(1-N^{-1})\left(-\frac{\kappa_\mathcal{G}(g_1)}{N}+|\nu|-1\right)N^{-1}\chi_{\underline{\mu},\underline{\nu}}(g_1)\\
&+N^{-1} \sum_{|\gamma|=\kappa_\mathcal{V}(\mu,\nu)=\kappa_\mathcal{V}(\mu,\gamma)} \chi_{\underline{\mu}_{n+|\gamma|},\gamma}(g_1)(\chi_{\nu}(y_1)-\chi_{\underline{\nu}_{|\nu|-1}}(y_1))\\
&+\sum_{\ell=0}^{|\nu|-2}  \sum_{|\gamma|=\kappa_\mathcal{V}(\mu,\nu)=\kappa_\mathcal{V}(\mu,\gamma)}N^{\ell-|\nu|} \chi_{\underline{\mu}_{n+|\gamma|},\gamma}(g_1)(\chi_{\underline{\nu}_{\ell+1}}(y_1)-\chi_{\underline{\nu}_{\ell}}(y_1))\\
&-N^{-1}\sum_{\ell=0}^{|\nu|-2}  \sum_{|\gamma|=\kappa_\mathcal{V}(\underline{\mu},\underline{\nu})=\kappa_\mathcal{V}(\mu,\gamma)}N^{\ell-|\nu|+1} \chi_{\underline{\mu}_{n+|\gamma|},\gamma}(g_1)(\chi_{\underline{\nu}_{\ell+1}}(y_1)-\chi_{\underline{\nu}_{\ell}}(y_1)). 
\end{align*}
In the last line we are using that if $|\gamma|=\kappa_\mathcal{V}(\underline{\mu},\underline{\nu})$ then $|\gamma|<|\mu|$ and $\kappa_\mathcal{V}(\mu,\gamma)=\kappa_\mathcal{V}(\underline{\mu},\gamma)$. If $t(\mu)=t(\nu)$ then $\kappa_\mathcal{V}(\mu,\nu)=\kappa_{\mathcal{V}}(\underline{\mu},\underline{\nu})$ and the last two terms cancel each other. We proceed with the remaining sums:
\begin{align*}
=&(1-N^{-1})\left(-\frac{\kappa_\mathcal{G}(g_1)}{N}+|\nu|\right)\chi_{\mu,\nu}(g_1)\\
&-(1-N^{-1})\left(-\frac{\kappa_\mathcal{G}(g_1)}{N}+|\nu|-1\right)N^{-1}\chi_{\underline{\mu},\underline{\nu}}(g_1)\\
&+N^{-1} \sum_{|\gamma|=\kappa_\mathcal{V}(\mu,\nu)=\kappa_\mathcal{V}(\mu,\gamma)} \chi_{\underline{\mu}_{n+|\gamma|},\gamma}(g_1)(\chi_{\nu}(y_1)-\chi_{\underline{\nu}_{|\nu|-1}}(y_1))
\end{align*}
\begin{align*}
=&(1-N^{-1})\left(-\frac{\kappa_\mathcal{G}(g_1)}{N}+|\nu|-1+\frac{N}{N-1}\right)(\chi_{\mu,\nu}-N^{-1}\chi_{\underline{\mu},\underline{\nu}})\\
&-N^{-1}\chi_{\mu,\nu}+N^{-1} \sum_{|\gamma|=\kappa_\mathcal{V}(\mu,\nu)=\kappa_\mathcal{V}(\mu,\gamma)} \chi_{\underline{\mu}_{n+|\gamma|},\gamma}(g_1)\chi_{\nu}(y_1)\\
&+N^{-1}\chi_{\underline{\mu},\underline{\nu}}-N^{-1} \sum_{|\gamma|=\kappa_\mathcal{V}(\mu,\nu)=\kappa_\mathcal{V}(\mu,\gamma)} \chi_{\underline{\mu}_{n+|\gamma|},\gamma}(g_1)\chi_{\underline{\nu}}(y_1)\\
&=(1-N^{-1})\left(-\frac{\kappa_\mathcal{G}(g_1)}{N}+|\nu|+\frac{1}{N-1}\right)(\chi_{\mu,\nu}-N^{-1}\chi_{\underline{\mu},\underline{\nu}}).
\end{align*}
Since $B\mathrm{e}_{\mu,\nu}=0$ if $t(\mu)=t(\nu)$ the theorem follows in this case. We now consider the case of two words $\mu,\nu\in \mathcal{V}_N$ with $t(\mu)\neq t(\nu)$. In this case, we simply note that Equation \eqref{tchimunu} implies that 
\begin{align*}
(1-N^{-1})T\chi_{\mu,\nu}=&\sum_{\ell=0}^{|\nu|-1}  \sum_{|\gamma|=\kappa_\mathcal{V}(\mu,\gamma)=|\nu|}N^{\ell-|\nu|} \chi_{\underline{\mu}_{n+|\gamma|},\gamma}(g_1)(\chi_{\underline{\nu}_{\ell+1}}(y_1)-\chi_{\underline{\nu}_{\ell}}(y_1))\\
&+(1-N^{-1})\left(-\frac{\kappa_\mathcal{G}(g_1)}{N}+|\nu|\right)\chi_{\mu,\nu}(g_1)\\
=&-\sum_{m\neq t(\mu)}\sum_{\ell=0}^{|\nu|-1}\sum_{|\gamma|=|\nu|-1, \, |\gamma\vee \nu|=\ell} N^{\ell-|\nu|} \chi_{\mu, \gamma m}(g_1)\\
&+(1-N^{-1})\left(-\frac{\kappa_\mathcal{G}(g_1)}{N}+|\nu|\right)\chi_{\mu,\nu}(g_1)\\
=&(1-N^{-1})\left(-\frac{\kappa_\mathcal{G}(g_1)}{N}+|\nu|-B\right)\chi_{\mu,\nu}(g_1)
\end{align*}

The only case left to consider is $\mu=\nu=\emptyset$ which holds trivially. This proves the theorem.
\end{proof}

\section{A spectral triple on $O_N$}
\label{sec:5}

We are now ready to assemble our operators into spectral triples on $O_N$. In \cite{GMR} a spectral triple was constructed by defining the operator $D_\kappa=(2P_\mathcal{F}-1)|c|-\kappa$. We proceed similarly and define $D$ as the closure in $L^{2}(O_{N})$ of the operator $(2P_\mathcal{F}-1)|c|-T$ with initial domain $C_{c}^{\infty}(\mathcal{G}_{N})$. In the same way, the operator $\tilde{D}$ is defined as the closure of $$(2P_\mathcal{F}-1)|c|-T-B+(N-1)^{-1}Q : C_{c}^{\infty}(\mathcal{G}_{N})\to L^{2}(O_{N}).$$ 
We note that $\tilde{D}=(2P_\mathcal{F}-1)|c|-\tilde{T}$ where 
\begin{equation}
\label{tzero}
\tilde{T}\mathrm{e}_{\mu,\nu}=\left(-\frac{\kappa_{\mathcal{V}}(\mu,\nu)}{N}+|\nu|\right)\mathrm{e}_{\mu,\nu}.
\end{equation}
The following two propositions prove part 1 and 2 of Theorem \ref{mainthm}.
\begin{proposition}
The triples $(O_N,L^2(O_N),\tilde{D})$ and $(O_N,L^2(O_N),D)$ are spectral triples representing the class $\widehat{[1]}\in K^{1}(O_{N})$.
\end{proposition} 

\begin{proof}
Since $D-\tilde{D}$ is a bounded self-adjoint operator, $\mathrm{Dom }(D)=\mathrm{Dom }(\tilde{D})$ and $D$ defines a spectral triple if and only if $\tilde{D}$ does. It is easily verified from Corollary \ref{kapakapa} and Equation \eqref{tzero} that $(i\pm \tilde{D})^{-1}$ is a compact operator. Moreover, the generators $S_i\in O_N$ preserve $\mathrm{Dom }(\tilde{D})$. The operator $\tilde{D}$ has bounded commutators with the generators $S_i$, which is seen from combining \cite[Theorem 3.19]{GMR} in the model over $\mathbb{C}$ and in the model over $C(\Omega_N)$; the latter gives bounded commutators with $\kappa_\mathcal{G}$ and the former bounded commutators with $D_\kappa=\tilde{D}+\frac{\kappa_\mathcal{G}}{N}$. Thus $D$ and $\tilde{D}$ define $K$-cycles for $O_{N}$. To identify their class in $K^{1}(O_{N})$, observe that the operator inequalities
\begin{equation}
\label{opineq}
0\leq (1-N^{-1})\kappa\leq \tilde{T}\leq \kappa,\quad 0\leq |c|+(1-N^{-1})\kappa\leq |c|+\tilde{T}\leq |c|+\kappa,
\end{equation}
hold true on the core $C_{c}^{\infty}(\mathcal{G}_{N})$ and hence on all of $\mbox{Dom }\tilde{T}=\mbox{Dom }\kappa$ as well as on $\mbox{Dom }(|c|+\tilde{T})=\mbox{Dom } (|c|+\kappa)$. By positivity, we have $\ker\kappa=\ker\tilde{T}=\mbox{im} P_{\mathcal{F}}$ and thus $\tilde{T}P_{\mathcal{F}}=P_{\mathcal{F}}\tilde{T}=0$ as well as
$$\ker(|c|+\tilde{T})=\ker |c|\cap \ker\tilde{T}=\ker |c| \cap \ker\kappa=\ker (|c|+\kappa).$$ 
We can thus write the operator $\tilde{D}$ and its phase $\tilde{D}|\tilde{D}|^{-1}$ as 
\[\tilde{D}=(2P_{\mathcal{F}}-1)|c|-\tilde{T}= (2P_{\mathcal{F}}-1)(|c|+\tilde{T}),\quad \tilde{D}|\tilde{D}|^{-1}=D_{\kappa}|D_{\kappa}|^{-1}=2P_{\mathcal{F}}-1,\]
which shows that $[D]=[\tilde{D}]=[D_{\kappa}]=\widehat{[1]}\in K^{1}(O_{N})$, by \cite{evansonon}, \cite{kaminkerputnam} and \cite[Theorem 3.19]{GMR}.
\end{proof}

\begin{proposition}The spectral triples $(O_N,L^2(O_N),\tilde{D})$ and $(O_N,L^2(O_N),D)$ are $\theta$-summable, $P_\mathcal{F}=\chi_{[0,\infty)}(\tilde{D})=\chi_{[0,\infty)}(D_\kappa)$ and the difference
$P_{\mathcal{F}}-\chi_{[0,\infty]}(D)$ is a finite rank operator. Moreover, for any $p>0$, the set 
$$\mathrm{Sum}^p(O_N,P_\mathcal{F}):=\{a\in O_N: [P_\mathcal{F},a]\in \mathcal{L}^p(L^2(O_N))\},$$
is a dense $*$-subalgebra of $O_N$. In particular, $(O_N, L^2(O_N), 2P_\mathcal{F}-1)$ is a p-summable generator of $K^1(O_N)$ for any $p>0$.
\end{proposition}

\begin{proof}
It suffices to prove $\theta$-summability for $\tilde{D}$. This follows since $\mathcal{H}_{n,k}$ is $N^{n+2k}$-dimensional and $-2(|n|+k)\leq \tilde{D}|_{\mathcal{H}_{n,k}}\leq 2(|n|+k)$.
The identity $P_\mathcal{F}=\chi_{[0,\infty)}(D_\kappa)$ follows from the construction using $P_\mathcal{F}L^2(O_N)=\ker(\kappa)$. Since $\tilde{D}|\tilde{D}|^{-1}=D_{\kappa}|D_{\kappa}|^{-1}$, we have $\chi_{[0,\infty)}(\tilde{D})=\chi_{[0,\infty)}(D_\kappa)$. Moreover, for $C=\|B\|$, we have the operator inequalities 
$$\tilde{D}|_{\mathcal{H}_{n,k}}-C\leq D|_{\mathcal{H}_{n,k}}\leq \tilde{D}|_{\mathcal{H}_{n,k}}+C, \quad\forall n,k.$$ 
Thus, by compactness of resolvents and the fact that the sign of $\tilde{D}|_{\mathcal{H}_{n,k}}$ is determined purely by $n$ and $k$, it follows that $\chi_{[0,\infty)}(D)-\chi_{[0,\infty)}(\tilde{D})$ is a finite rank operator. The remaining statements follow from \cite[Proposition 2.2.5]{GM}, which uses methods from \cite{extensionsgoffeng,kaminkerputnam}.
\end{proof}

\section{The Fr\"ohlich functionals on $O_N$}
\label{subsec:5.1}

In \cite{froehlicetal}, Fr\"ohlich et al associated a state with $\theta$-summable spectral triples. If $(A,\mathcal{H},D)$ is a $\theta$-summable spectral triple, one defines the state on $a\in A$ as
$$\phi^D_t(a):=\frac{\mathrm{tr}(a\mathrm{e}^{-tD^2})}{\mathrm{tr}(\mathrm{e}^{-tD^2})},\quad t>0.$$ 
The assumption that an extended limit of $\phi^D_t$ as $t\to 0$ is tracial is used in \cite{froehlicetal}. This is clearly an unrealistic assumption if $A$ admits no traces. In \cite[Corollary 12.3.5]{sukolord} it is shown that under certain finite dimensionality conditions on $(A, \mathcal{H}, D)$, which are slightly stronger than finite summability, extended limits of $\phi_t^D$ are tracial. This provides interesting connections between the tracial property of the states introduced by Fr\"ohlich et al and Connes' tracial obstructions to finite summability.

Since $D$, $D_0$ and $D_\kappa$ are $\theta$-summable, they allow for the definition of Fr\"ohlich functionals on $O_N$. For $t>0$ we define the following states on $\mathbb{B}(L^2(O_N))$:
$$\phi_t(T):=\frac{\mathrm{tr}(T\mathrm{e}^{-tD^2})}{\mathrm{tr}(\mathrm{e}^{-tD^2})}, \; \tilde{\phi}_t(T):=\frac{\mathrm{tr}(T\mathrm{e}^{-t\tilde{D}^2})}{\mathrm{tr}(\mathrm{e}^{-t\tilde{D}^2})} \quad\mbox{and}\quad \phi_t^\kappa(T):=\frac{\mathrm{tr}(T\mathrm{e}^{-tD_\kappa^2})}{\mathrm{tr}(\mathrm{e}^{-tD_\kappa^2})}.$$
A state $\omega\in L^\infty[0,1]^*$ is said to be an extended limit at $0$ if $\omega(f)=0$ whenever $f=0$ near $0$. For an extended limit $\omega$ at $0$, we define $\phi_\omega:=\omega\circ \phi_t$, $\tilde{\phi}_\omega:=\omega\circ \tilde{\phi}_t$ and $\phi_\omega^\kappa:=\omega\circ \phi_t^\kappa$. The next result proves part 3 of Theorem \ref{mainthm}.

\begin{proposition}
For any extended limit $\omega$, there exists probability measures $\tilde{m}_\omega$ and $m_\omega^\kappa$ on $\Omega_N$ such that for $a\in O_N$
$$\tilde{\phi}_\omega(a)=\int_{\Omega_N} \Phi(a)\mathrm{d}\tilde{m}_\omega\quad\mbox{and}\quad \phi_\omega^\kappa(a)=\int_{\Omega_N} \Phi(a)\mathrm{d}m_\omega^\kappa.$$
\end{proposition}

Using the fact that $\Phi(S_\mu S_\nu^*)=\delta_{\mu,\nu}S_\mu S_\mu^*$ the proposition is immediate from the next lemma which in turn is a computational exercise.

\begin{lemma}
For any finite words $\mu,\nu,\sigma,\rho\in \mathcal{V}_N$, 
\begin{align*}
\langle \mathrm{e}_{\mu,\nu}, &S_\rho S_\sigma^*\mathrm{e}_{\mu,\nu}\rangle_{L^2(O_N)}\\
&=
\begin{cases}
\delta_{\rho,\sigma}\delta_{|\sigma\vee \mu|,\min\{|\mu|,|\sigma|}N^{\min\{0,|\mu|-|\rho|\}}, \quad & t(\mu)\neq t(\nu),\\
\\
(N-1)^{-1}\delta_{\rho,\sigma}\bigg(\delta_{|\sigma\vee \mu|,\min\{|\mu|,|\sigma|}(N-2)N^{\min\{0,|\mu|-|\rho|\}}+\quad & t(\mu)= t(\nu)\\
\qquad\qquad\qquad\qquad\qquad\qquad\delta_{|\sigma\vee \underline{\mu}|,\min\{|\mu|-1,|\sigma|}N^{\min\{0,|\mu|-|\rho|-1\}}\bigg).\end{cases}.
\end{align*}
Here $\mu$ starts with $\underline{\mu}$ and $|\underline{\mu}|=|\mu|-1$.
\end{lemma}

\begin{remark}
It is reasonable to expect that $\phi_\omega$, $\tilde{\phi}_\omega$ and $\phi_\omega^\kappa$ are in fact related to the KMS state $\phi$. We have not been able to prove this. A simple induction procedure, or using uniqueness of KMS states on $O_N$ combined with \cite{lacanesh}, shows that it suffices to prove that $\phi_\omega(S_\nu S_\nu^*)=N\phi_\omega(S_{\nu j}S_{\nu j}^*)$ for any finite word $\nu\in \mathcal{V}_N$ and $j=1,\ldots, N$.
\end{remark}

\section*{Acknowledgement}
We thank the MATRIX for the program \emph{Refining $C^*$-algebraic invariants for dynamics using $KK$-theory} in Creswick, Australia (2016) where this work came into being. We are grateful to the support from Leibniz University Hannover where this work was initiated. We also thank Francesca Arici, Robin Deeley, Adam Rennie and Alexander Usachev for fruitful discussions and helpful comments, and the anonymous referee for a careful reading of the manuscript. The first author was supported by the Swedish Research Council Grant 2015-00137 and Marie Sklodowska Curie Actions, Cofund, Project INCA 600398.


\begin{thebibliography}{99.}

\bibitem{Connestrace} A. Connes, \emph{Compact metric spaces, Fredholm modules and hyperfiniteness}, Ergod. Th. Dyn. Sys. 9 (1989), p. 207--230.

\bibitem{connesbook} A.~Connes, \emph{Noncommutative Geometry}, Academic Press, London, 1994.

\bibitem{cmtwisted} A. Connes, H. Moscovici, \emph{Type III and spectral triples}, Traces in number theory, geometry and quantum fields, p. 57--71, Aspects Math., E38, Friedr. Vieweg, Wiesbaden, 2008.

\bibitem{Coornaert} M. Coornaert, \emph{Mesures de Patterson-Sullivan sur le bord d'un espace hyperbolique au sens de Gromov}, Pac. J. Math. 159 (2) (1993), 241-270.

\bibitem{CMRV} G. Cornelissen, M.Marcolli, K. Reihani, A Vdovina, \emph{Noncommutative geometry on trees and buildings}, "Traces in Geometry, Number Theory, and Quantum Fields", pp. 73-98, Vieweg Verlag, 2007.

\bibitem{thealgebraon} J. Cuntz, \emph{Simple $C^*$-algebras generated by isometries}, Comm. Math. Phys. 57 (1977), no. 2, p. 173--185. 

\bibitem{emersonnica} H. Emerson, B. Nica, \emph{$K$-homological finiteness and hyperbolic groups}, to appear in J. Reine Angew. Math.

\bibitem{evansonon} D. E. Evans, \emph{On $O_n$}, Publ. Res. Inst. Math. Sci. 16 (1980), no. 3, p. 915--927. 

\bibitem{gillaspyetal} C. Farsi, E Gillaspy, A. Julien, S. Kang, J. Packer, \emph{Wavelets and spectral triples for fractal representations of Cuntz algebras}, arXiv:1603.06979.

\bibitem{froehlicetal} J. Fr\"ohlich, O. Grandjean, A. Recknagel, \emph{Supersymmetric quantum theory and differential geometry}, Comm. Math. Phys. 193 (1998), no. 3, p. 527--594. 

\bibitem{extensionsgoffeng} M. Goffeng, \emph{Equivariant extensions of $*$-algebras}, New York J. Math. 16 (2010), p. 369--385. 

\bibitem{GM} M. Goffeng, B. Mesland, \emph{Spectral triples and finite summability on Cuntz-Krieger algebras}, Doc. Math., {\bf 20} (2015),  89--170.

\bibitem{GMR} M. Goffeng, B. Mesland, A. Rennie, \emph{Shift tail equivalence and an unbounded representative of the Cuntz-Pimsner extension}, to appear in Ergod. Th. Dyn. Sys.

\bibitem{kaminkerputnam} J. Kaminker, I. Putnam, \emph{$K$-theoretic duality of shifts of finite type}, Comm. Math. Phys. 187 (1997), no. 3, p. 509--522

\bibitem{kasparov} G.G. Kasparov, \emph{The operator $K$-functor and extensions of $C\sp{\ast} $-algebras}, Izv. Akad. Nauk SSSR Ser. Mat. 44 (1980), no. 3, p. 571--636, 719.

\bibitem{lacanesh} M. Laca and S. Neshveyev, {\em KMS states of quasi-free dynamics on Pimsner algebras}, J. Funct. Anal. {\bf 211} (2004), p. 457--482.

\bibitem{sukolord} S. Lord, F. Sukochev, D. Zanin, \emph{Singular traces. Theory and applications.} de Gruyter Studies in Mathematics 46. De Gruyter, Berlin, 2013.

\bibitem{OP}  D. Olesen and G. K. Pedersen, \emph{Some $C^*$-dynamical systems with a single KMS state}, Math. Scand. \textbf{42} (1978), 111--118.

\bibitem{paskrennie} D. Pask, A. Rennie, \emph{The noncommutative geometry of graph $C\sp *$-algebras. I. The index theorem}, J. Funct. Anal. 233 (2006), no. 1, p. 92--134. 

\bibitem{Ren1} J. Renault, \emph{A groupoid approach to $C^{*}$-algebras}, LNM 793, Springer 1980.

\bibitem{Ren2} J. Renault, \emph{Cuntz-like algebras}, Operator theoretical methods (Timisoara, 1998) Theta Found, Bucharest, 2000, p. 371--386.

\end{thebibliography}
\end{document}